\documentclass[fleqn,12pt,twoside]{article}

\usepackage[headings]{espcrc1}
\readRCS
$Id: espcrc1.tex,v 1.2 2004/02/24 11:22:11 spepping Exp $
\usepackage{graphicx}
\usepackage[figuresright]{rotating}

\newtheorem{theorem}{Theorem}[section]

\newtheorem{claim}{Claim}

\newtheorem{conjecture}[theorem]{Conjecture}
\newtheorem{corollary}[theorem]{Corollary}

\newtheorem{definition}{Definition}

\newtheorem{lemma}[theorem]{Lemma}

\newtheorem{proposition}[theorem]{Proposition}

\newenvironment{proof}[1][Proof.]{\begin{trivlist}
\item[\hskip \labelsep {\bfseries #1}]}{\end{trivlist}}

\newenvironment{acknowledgement}[1][Acknowledgement]{\begin{trivlist}
\item[\hskip \labelsep {\bfseries #1}]}{\end{trivlist}}

\newcommand{\AmS}{{\protect\the\textfont2
  A\kern-.1667em\lower.5ex\hbox{M}\kern-.125emS}}
	
	\usepackage{xcolor}
\definecolor{blau}{rgb}{0.1,0.0,0.9}
\definecolor{gruen}{cmyk}{1.0,0.2,0.7,0.07}
\definecolor{mag}{cmyk}{0.0,0.9,0.3,0.0}

\usepackage{amsmath}
\usepackage{amsfonts}

\usepackage{tikz}
\usetikzlibrary{arrows}
\usetikzlibrary{shapes}
\usetikzlibrary{decorations.markings}

\title{Cyclic Deficiency of Graphs}

\author{Armen S. Asratian\address[MCSD]{Department of Mathematics, Link\"oping University,\\
SE-581 83 Link\"oping, Sweden}\thanks{email: armen.asratian@liu.se},
    Carl Johan Casselgren\addressmark[MCSD]\thanks{email: carl.johan.casselgren@liu.se},
    Petros A. Petrosyan\address{Department of Informatics and Applied Mathematics,\\
Yerevan State University, 0025, Armenia}%
\address{Institute for Informatics and Automation Problems,\\
National Academy of Sciences, 0014, Armenia}%
\thanks{email: petros\_petrosyan@ysu.am, pet\_petros@ipia.sci.am}}

\runtitle{Cyclic Deficiency of Graphs}\runauthor{Armen S. Asratian, Carl Johan Casselgren, Petros A. Petrosyan}

\begin{document}

\maketitle

\begin{abstract}
A proper edge  coloring of a graph $G$ with colors $1,2,\dots,t$ is
called a \emph{cyclic interval $t$-coloring} if for each vertex $v$
of $G$ the edges incident to $v$ are colored by consecutive colors,
under the condition that color $1$ is considered as consecutive to
color $t$. 
In this paper we introduce and investigate a new notion, the {\em cyclic deficiency}
of a graph $G$, defined as the minimum number of pendant edges whose
attachment to $G$ yields a graph admitting a cyclic interval coloring; this number can be considered as a measure of closeness of $G$ of being cyclically interval colorable.
We determine or bound the cyclic deficiency of several
 families
of graphs. 
In particular, we present
examples of graphs of bounded maximum degree
with arbitrarily large cyclic deficiency, and graphs whose cyclic
deficiency approaches the number of vertices.
Finally, we conjecture that the cyclic deficiency of any graph does not
exceed the number of vertices, and we present several results
supporting this conjecture. 
\\

\end{abstract}

		\bigskip

\section{Introduction}

 We use \cite{West} for terminology and notation not defined here. 
All graphs in this paper are
finite, undirected, and contain no multiple edges or loops;
a multigraph may have both multiple edges and loops. 
   $V(G)$ and $E(G)$ denote the sets of vertices and
edges of a graph $G$, respectively.
A proper $t$-edge coloring of a graph $G$ is a
    mapping $\alpha:E(G)\longrightarrow  \{1,\dots,t\}$ such that
    $\alpha(e)\not=\alpha(e')$ for every pair of adjacent
    edges $e$ and $e'$ in $G$.
  
A proper $t$-edge coloring of a graph $G$ is called an 
\emph{interval $t$-coloring} if the colors of the edges incident to every
vertex $v$ of $G$ form an interval of integers.
 This notion was introduced by 
Asratian and Kamalian \cite{AsrKam} 
(available in English as \cite{AsrKamJCTB}), 
motivated by the problem of constructing timetables
 without ``gaps'' for teachers and classes.
Generally, it is an NP-complete problem
to determine whether a bipartite graph
has an interval coloring \cite{Seva}. However some classes of 
graphs have been proved to admit interval colorings; it is known,
for example, that trees, regular and complete bipartite graphs
\cite{AsrKam,Hansen,Kampreprint},  doubly convex bipartite graphs
\cite{AsrDenHag,KamDiss}, grids \cite{GiaroKubale1}, 
and 
outerplanar bipartite graphs \cite{GiaroKubale2} have interval colorings. 
Additionally, all $(2,b)$-biregular graphs \cite{Hansen,HansonLotenToft,KamMir} and
 $(3,6)$-biregular graphs \cite{CarlJToft} admit interval colorings, where an
\emph{$(a,b)$-biregular} graph is a bipartite graph
where the vertices in one part all have degree $a$ and the vertices
in the other part all have degree $b$.

In \cite{GiaroKubaleMalaf1,GiaroKubaleMalaf2} Giaro et al. introduced and investigated 
the {\em deficiency} $\text{def}(G)$ of a graph $G$ defined
as the minimum number
of pendant edges whose attachment to $G$ makes the resulting graph
 interval colorable.
In \cite{GiaroKubaleMalaf1} it was proved that
there are bipartite graphs 
whose deficiency approaches the number of vertices,
and in
\cite{GiaroKubaleMalaf2} it was
proved that if $G$ is an $r$-regular graph with an odd number of vertices,
then $\text{def}(G) \geq \frac{r}{2}$. Furthermore in \cite{GiaroKubaleMalaf2} the deficiency of 
complete graphs, wheels and broken wheels was determined. 
Schwartz \cite{Schwartz} obtained tight bounds on the deficiency of some
families of regular graphs.
Recently, Petrosyan and Khachatrian \cite{PetrosHrant}
proved that for near-complete graphs
$\mathrm{def}(K_{2n+1}-e)=n-1$ (where $e$ is an edge of $K_{2n+1}$),
thereby confirming a conjecture of Borowiecka-Olszewska et al. 
\cite{B-OD-BHal}.
Further results on deficiency appear in
\cite{AltinCaporHertz,B-OD-BHal,B-OD-B,BouchHertzDesau,FengHuang,KhachOuterplanar,PetrosHrant}. 



Another type of proper $t$-edge colorings, a \emph{cyclic interval
$t$-coloring}, was introduced by de Werra and Solot
\cite{deWerraSolot}. A proper $t$-edge coloring
$\alpha:E(G)\longrightarrow  \{1,\dots,t\}$ of a graph $G$  is
called a \emph{cyclic interval $t$-coloring} if the colors of the edges
incident to every vertex $v$ of $G$ either form an interval of
integers or the set $\{1,\ldots,t\}\setminus \{\alpha(e): \text{$e$
is incident to $v$}\}$ is an interval of integers. 
This notion was
motivated by scheduling problems arising in flexible manufacturing
systems, in particular the so-called \emph{cylindrical open shop
scheduling problem}. Clearly, any interval $t$-coloring of a graph
$G$ is also a cyclic interval $t$-coloring. Therefore all above
mentioned classes of graphs which admit interval edge colorings,
also admit cyclic interval colorings. 
 
Generally, the problem of
determining whether a given bipartite graph admits a cyclic interval
coloring is $NP$-complete \cite{KubaleNadol}. 
Nevertheless,
all graphs with maximum degree at most $3$ \cite{Nadol}, complete multipartite graphs
\cite{AsratianCasselgrenPetrosyan}, outerplanar bipartite graphs
\cite{deWerraSolot}, bipartite graphs with maximum degree $4$ 
and Eulerian bipartite graphs with maximum degree not exceeding $8$
\cite{AsratianCasselgrenPetrosyan}, 
and some families of biregular bipartite graphs 
\cite{CarlJToft,CasselgrenPetrosyanToft,AsratianCasselgrenPetrosyan}
admit cyclic interval colorings.

\medskip

In this paper we introduce  and investigate a new notion,
the {\em cyclic deficiency}
 of a graph $G$, denoted by $\text{def}_c(G)$, defined as the minimum number of
pendant edges that has to be attached  to $G$ in order to obtain a  graph admitting a cyclic interval coloring.
This number can be considered as a measure of closeness of a graph $G$ of being cyclically interval colorable.
Clearly, $\text{def}_c(G) =0$ if and only if $G$
admits a cyclic interval coloring.
Hence, there are infinite families of graphs with cyclic deficiency
$0$ but arbitrarily large deficiency, e.g. regular graphs
with an odd number of vertices.

We present examples of graphs of bounded maximum degree
with arbitrarily large cyclic deficiency, and graphs whose cyclic
deficiency approaches the number of vertices. 
We determine or bound the cyclic deficiency of several families
of graphs 
and describe three different methods
 for constructing graphs with large cyclic deficiency.
Our constructions generalize earlier constructions by
Hertz, Erd\H{o}s,
Malafiejski 
and Sevastjanov
(see. e.g. \cite{GiaroKubaleMalaf1,JensenToft,GiaroKubaleMalaf2}).
Since the inequality $\text{def}_c(G) \leq \text{def}(G)$
holds for any graph $G$,
our constructions yield new classes of graphs with large deficiency.
Finally, we conjecture that the cyclic deficiency of any graph does not
exceed the number of vertices, and we present several results
supporting this conjecture. In particular we prove that the conjecture is true 
for graphs with maximum degree at most $5$,
bipartite graphs with maximum degree at most $8$, and graphs 
where the difference between 
maximum and minimum degrees does not exceed $2$.

\bigskip

\section{Definitions and auxiliary results}      

	A graph $G$ is \emph{cyclically interval
colorable} if it has a cyclic interval $t$-coloring for some
positive integer $t$. The set of all cyclically interval colorable
graphs is denoted by $\mathfrak{N}_{c}$. For a graph $G\in
\mathfrak{N}_{c}$, the maximum number of colors in a cyclic interval coloring of 
$G$ is denoted by $W_{c}(G)$.

The degree of a 
	vertex $v$ of a graph  $G$ is denoted by $d_G(v)$.
	 $\Delta(G)$ and $\delta(G)$ denote the maximum and minimum
degrees of $G$, respectively. A graph $G$ is \emph{even} if the degree of every vertex of $G$ is
even. The diameter of a graph $G$ we denote by $\mathrm{diam}(G)$.

We shall need a classic result from factor theory.
A \emph{$2$-factor} of a multigraph $G$ (where loops are allowed) is a
$2$-regular spanning subgraph of $G$.
%
\begin{theorem} (Petersen's Theorem).
\label{th:Petersen}
 Let $G$ be a $2r$-regular
multigraph (where loops are allowed). Then $G$ has a decomposition
into edge-disjoint $2$-factors.
\end{theorem}

The {\em chromatic index $\chi'(G)$} of a graph $G$
	is the minimum number  $t$ for which there exists  a proper
	$t$-edge coloring of $G$. 
	
	
	\begin{theorem} (Vizing's Theorem).
\label{th:Vizing}
For any 
graph $G$, $\chi'(G) = \Delta(G)$
	or $\chi'(G) = \Delta(G) +1$. 
	\end{theorem}
	A graph $G$  is said to be \emph{Class $1$} if $\chi'(G) = \Delta(G)$, and  \emph{Class $2$}
	 if $\chi'(G) = \Delta(G) +1$.

The next result gives a sufficient condition for a graph to be Class 1 (see, for example,  \cite{Fournier}).
\begin{theorem}
\label{th:Fournier}
	If $G$ is a graph where no two vertices of maximum degree are adjacent,
	then $G$ is Class 1.
\end{theorem}
We denote by $\mathbb{N}$ the the set of all positive integers. 
If $t\in \mathbb{N}$ and $A$ is a subset of the set $\{1,\ldots,t\}$, then $A$ is called a {\em cyclic interval modulo $t$} if either $A$ or $\{1,\dots,t\} \setminus A$
is an interval of integers. The {\em deficiency modulo $t$}
of a given set $A$ of integers is the minimum size of a set
$B$ of positive integers such that $A \cup B$ is a cyclic interval
modulo $t$. A set of positive integers $A$ is called {\em near-cyclic modulo $t$}
(or just {\em near-cyclic}) if there is an integer $k$ such that $A \cup \{k\}$ is a cyclic
interval modulo $t$.

If $\alpha $ is a proper edge coloring of $G$ and $v\in V(G)$, then
$S_{G}\left(v,\alpha \right)$ (or $S\left(v,\alpha \right)$) denotes
the set of colors appearing on edges incident to $v$.


\begin{definition}
	{\em For a given graph $G$ and a proper $t$-edge coloring $\alpha$ of $G$,
	the {\em cyclic deficiency of $\alpha$ at a vertex} $v \in V(G)$, denoted
	$\text{def}_c(v,\alpha)$ is the deficiency modulo $t$ of the set $S\left(v,\alpha \right)$. The cyclic deficiency $\text{def}_c(G,\alpha)$
	of the edge coloring $\alpha$ is the sum of cyclic 
	deficiencies of all vertices in $G$.
	The  cyclic deficiency $\text{def}_c(G)$ of $G$
	can then be defined as
	the minimum of $\text{def}_c(G,\alpha)$ taken over all proper 
	edge colorings $\alpha$ of $G$.}
\end{definition}

\begin{figure}[h]
\begin{center}
\includegraphics[height=17pc]{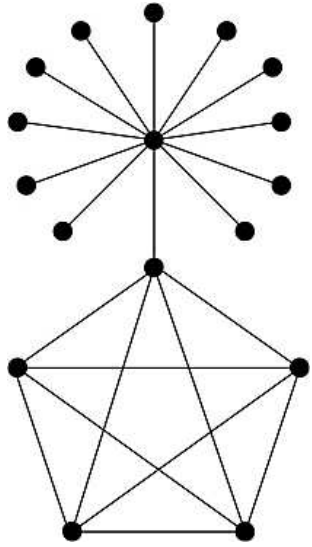}\\
\caption{A cyclically interval non-colorable graph with $17$ vertices.}
\label{fig}
\end{center}
\end{figure}

Clearly, $\mathrm{def}_c(G)=0$ if and only if $G$ has a cyclic interval coloring. 
In particular, this implies that the problem of computing the cyclic deficiency of a given graph is $NP$-complete. 

In \cite{Kubale} it was shown that among all connected graphs with at most $6$ vertices there are only seven Class 1 graphs without interval colorings. On the other hand, in \cite{BeinekeWilson} it was shown that among all connected graphs with at most $6$ vertices there are only eight Class 2 graphs. For all these graphs we constructed cyclic interval colorings, so all connected graphs with at most $6$ vertices
are cyclically interval colorable. However, there exists a connected graph 
 with $17$ vertices that has no cyclic interval coloring \cite{PetMkhitaryan}. (See Figure \ref{fig})\\

%

Finally, we need the notion of a {\em projective plane}.

\begin{definition}\label{def:project}
{\em A {\em finite projective
plane $\pi(n)$ of order $n$} ($n\geq 2$) has $n^{2}+n+1$ points and
$n^{2}+n+1$ lines, and satisfies the following properties:

\begin{description}
\item[P1] any two points determine a line;

\item[P2] any two lines determine a point;

\item[P3] every point is incident to $n+1$ lines;

\item[P4] every line is incident to $n+1$ points.
\end{description}
}
\end{definition}

\bigskip

\section{Comparison of deficiency and cyclic deficiency}\


Clearly, $\mathrm{def}_{c}(G) \leq \mathrm{def}(G)$ for every graph $G$, 
since any interval coloring of $G$ is also
a cyclic interval coloring of $G$.
In particular, $\mathrm{def}_{c}(G)=\mathrm{def}(G)=0$ for every graph $G$ which admits an interval coloring. 


In this section we will show that the difference between 
the deficiency and cyclic deficiency
can be arbitrarily large, even for graphs with large deficiency.
 
\begin{figure}[h]
\begin{center}
\includegraphics[width=20pc]{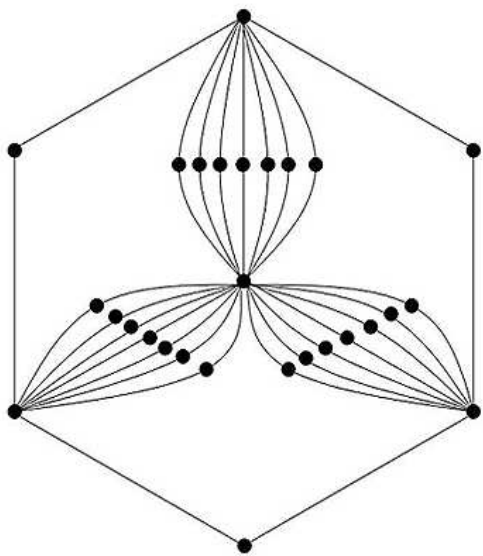}\\
\caption{The graph $S_{7,7,7}$.}\label{fig1}
\end{center}
\end{figure}
 
%
We begin by considering generalizations of
two families of bipartite graphs with 
large deficiency introduced by Giaro et al. \cite{GiaroKubaleMalaf1}. 
%
For any $a,b,c\in \mathbb{N}$, define the graph $S_{a,b,c}$ as follows:
\begin{center}
$V(S_{a,b,c})=\{u_{0},u_{1},u_{2},u_{3},v_{1},v_{2},v_{3}\}\cup \{x_{1},\ldots, x_{a},y_{1},\ldots, y_{b},z_{1},\ldots,z_{c}\}$ and

\medskip

$E(S_{a,b,c})=\{u_{1}v_{1},v_{1}u_{2},u_{2}v_{2},v_{2}u_{3},u_{3}v_{3},v_{3}u_{1}\}\cup \{u_{0}x_{i},u_{1}x_{i}:1\leq i\leq a\}$\\
$\cup\{u_{0}y_{j},u_{2}y_{j}:1\leq j\leq b\}\cup\{u_{0}z_{k},u_{3}z_{k}:1\leq k\leq c\}$.
\end{center}
Figure \ref{fig1} shows the graph $S_{7,7,7}$.

Next we define a family of graphs $M_{a,b,c}$  ($a,b,c\in \mathbb{N}$).
We set
\begin{center}
$V(M_{a,b,c})=\{u_{0},u_{1},u_{2},u_{3}\}\cup \{x_{1},\ldots, x_{a},y_{1},\ldots, y_{b},z_{1},\ldots,z_{c}\}$ and 

$E(M_{a,b,c})=\{u_{0}x_{i},u_{1}x_{i},u_{2}x_{i}:1\leq i\leq a\}\cup\{u_{0}y_{j},u_{2}y_{j},u_{3}y_{j}:1\leq j\leq b\}$\\
$\cup\{u_{0}z_{k},u_{3}z_{k},u_{1}z_{k}:1\leq k\leq c\}$.
\end{center}
Figure \ref{fig2} shows the graph $M_{5,5,5}$.

Clearly, $S_{a,b,c}$ and $M_{a,b,c}$ are connected bipartite graphs. 
Giaro et al.  \cite{GiaroKubaleMalaf1} showed that 
the graphs $S_k=S_{k,k,k}$ and $M_k=M_{k,k,k}$ 
satisfy
$\mathrm{def}(S_{k})\geq k-6$ and $\mathrm{def}(M_{k})\geq k-4$ 
for each $k\geq 6$;
that is, the deficiencies of $S_{k}$ and $M_{k}$ 
grow with the number of vertices. By contrast, we prove that all graphs
in the families $\{S_{a,b,c}\}$ and $\{M_{a,b,c}\}$ have cyclic deficiency $0$.

\begin{figure}[h]
\begin{center}
\includegraphics[width=25pc]{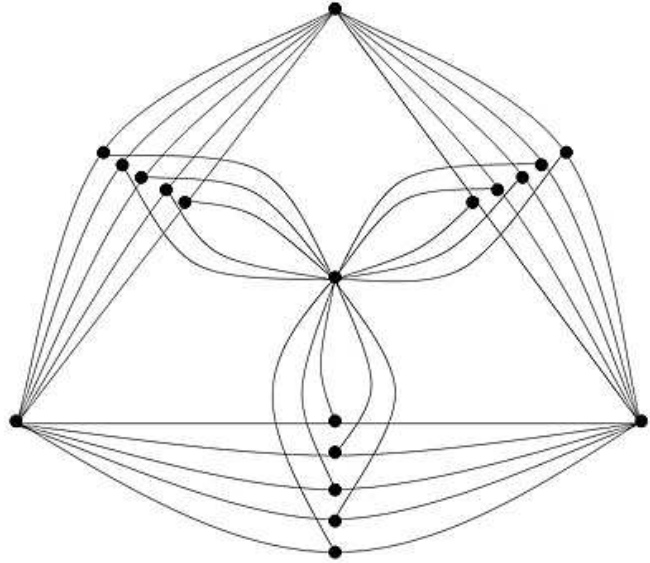}\\
\caption{The graph $M_{5,5,5}$.}\label{fig2}
\end{center}
\end{figure}

\begin{theorem}
\label{mytheorem2.1.1} For any $a,b,c\in \mathbb{N}$, $\mathrm{def}_{c}(S_{a,b,c})=\mathrm{def}_{c}(M_{a,b,c})=0$.
\end{theorem}
\begin{proof}
For the proof, we construct edge-colorings of the
graphs $S_{a,b,c}$ and $M_{a,b,c}$ with cyclic deficiency zero.
%
We first construct an edge-coloring $\alpha$ of the graph $S_{a,b,c}$. We define this coloring as follows:

\begin{description}
\item[($1$)] for $1\leq i\leq a$, let
$\alpha \left(u_{0}x_{i}\right)=i$ and $\alpha \left(u_{1}x_{i}\right)=i+1$;

\item[($2$)] for $1\leq j\leq b$, let
$\alpha \left(u_{0}y_{j}\right)=a+j$ and $\alpha \left(u_{2}y_{j}\right)=a+1+j$;

\item[($3$)] for $1\leq k\leq c$, let
$\alpha \left(u_{0}z_{k}\right)=a+b+k$;

\item[($4$)] for $1\leq k\leq c-1$, let
$\alpha \left(u_{3}z_{k}\right)=a+b+1+k$, and $\alpha \left(u_{3}z_{c}\right)=1$;

\item[($5$)] $\alpha \left(u_{1}v_{1}\right)=a+2,\alpha \left(v_{1}u_{2}\right)=a+1,\alpha \left(u_{2}v_{2}\right)=a+b+2,\alpha \left(v_{2}u_{3}\right)=a+b+1,$\\
$\alpha \left(u_{3}v_{3}\right)=2,\alpha \left(v_{3}u_{1}\right)=1$.
\end{description}

It is not difficult to see that if $a=b=c=1$, then $\alpha$ is an interval $4$-coloring of $S_{1,1,1}$;
otherwise $\alpha$ is a cyclic interval $(a+b+c)$-coloring of $S_{a,b,c}$.


Next we define an edge-coloring $\beta$ of the graph $M_{a,b,c}$ as follows: 

\begin{description}
\item[($1^{\prime}$)] for $2\leq i\leq a$, let
$\beta \left(u_{0}x_{i}\right)=i-1$, and $\beta \left(u_{0}x_{1}\right)=a+b+c+1$;

\item[($2^{\prime}$)] for $1\leq i\leq a$, let
$\beta \left(u_{1}x_{i}\right)=i$ and $\beta \left(u_{2}x_{i}\right)=i+1$;

\item[($3^{\prime}$)] for $1\leq j\leq b$, let
$\beta \left(u_{0}y_{j}\right)=a-1+j$;

\item[($4^{\prime}$)] for $1\leq j\leq b$, let
$\beta \left(u_{2}y_{j}\right)=a+1+j$ and $\beta \left(u_{3}y_{j}\right)=a+j$;

\item[($5^{\prime}$)] for $1\leq k\leq c$, let
$\beta \left(u_{0}z_{k}\right)=a+b-1+k$;

\item[($6^{\prime}$)] for $1\leq k\leq c$, let
$\beta \left(u_{3}z_{k}\right)=a+b+k$ and $\beta \left(u_{1}z_{k}\right)=a+b+1+k$.
\end{description}
It is easy to verify that $\beta$ is a cyclic interval 
$(a+b+c+1)$-coloring of $M_{a,b,c}$. 
%
%
Hence, $\mathrm{def}_{c}(S_{a,b,c})=\mathrm{def}_{c}(M_{a,b,c})=0$. 
~$\square$
\end{proof}


\begin{figure}
\centering
\begin{minipage}{.5\textwidth}
  \centering
  \begin{tikzpicture}[style=thick]
   \def\CIRCLEINNERRAD{1.5}
   \def\TAILLENGTH{3}
      \coordinate (U5lx) at (90:\TAILLENGTH cm);
      \coordinate (U5l0) at (90:\CIRCLEINNERRAD cm);
      \coordinate (U5l1) at (162:\CIRCLEINNERRAD cm);
      \coordinate (U5l2) at (-126:\CIRCLEINNERRAD cm);
      \coordinate (U5l3) at (-54:\CIRCLEINNERRAD cm);
      \coordinate (U5l4) at (18:\CIRCLEINNERRAD cm);

      
      \draw[fill=black] (U5lx) circle (2pt);
      \draw[fill=black] (U5l0) circle (2pt);
      \draw[fill=black] (U5l1) circle (2pt);
      \draw[fill=black] (U5l2) circle (2pt);
      \draw[fill=black] (U5l3) circle (2pt);
	  \draw[fill=black] (U5l4) circle (2pt);

      \draw (U5l0) --  (U5lx);
      \draw (U5l0) --  (U5l1) --  (U5l2) --  (U5l3) --  (U5l4) --  (U5l0)  --  (U5l2) --  (U5l4) --  (U5l1) --  (U5l3) --  (U5l0);

  \end{tikzpicture}\\
  \caption{The graph $H$.}
  \label{fig:test1}
\end{minipage}%
\begin{minipage}{.5\textwidth}
  \centering
  \begin{tikzpicture}[style=thick]
   \def\CIRCLEINNERRAD{2}
   \def\TAILLENGTH{3.2}
  	
      \coordinate (U9ry) at (80:3cm);
      \coordinate (U9rx) at (90:\TAILLENGTH cm);
      \coordinate (U9r0) at (90:\CIRCLEINNERRAD cm);
	  \coordinate (U9r1) at (130:\CIRCLEINNERRAD cm);
      \coordinate (U9r2) at (170:\CIRCLEINNERRAD cm);
      \coordinate (U9r3) at (-150:\CIRCLEINNERRAD cm);
      \coordinate (U9r4) at (-110:\CIRCLEINNERRAD cm);
      \coordinate (U9r5) at (-70:\CIRCLEINNERRAD cm);
      \coordinate (U9r6) at (-30:\CIRCLEINNERRAD cm);
      \coordinate (U9r7) at (10:\CIRCLEINNERRAD cm);
      \coordinate (U9r8) at (50:\CIRCLEINNERRAD cm);

	  \draw[fill=black] (U9rx) circle (2pt);
      \draw[fill=black] (U9ry) circle (2pt);
      \draw[fill=black] (U9r0) circle (2pt);
      \draw[fill=black] (U9r1) circle (2pt);
      \draw[fill=black] (U9r2) circle (2pt);
      \draw[fill=black] (U9r3) circle (2pt);
      \draw[fill=black] (U9r4) circle (2pt);
      \draw[fill=black] (U9r5) circle (2pt);
      \draw[fill=black] (U9r6) circle (2pt);
      \draw[fill=black] (U9r7) circle (2pt);
      \draw[fill=black] (U9r8) circle (2pt);
      
      \draw (U9r0) --  (U9rx);
      \draw (U9r0) --  (U9ry);
      
      \draw (U9r0) --  (U9r1) --  (U9r2) --  (U9r3) --  (U9r4) --  (U9r5) --  (U9r6) --  (U9r7) --  (U9r8) --  (U9r0)  --  (U9r2) --  (U9r4) --  (U9r6) --  (U9r8); \draw (U9r1) --  (U9r3) --  (U9r5) --  (U9r7) --  (U9r0) --  (U9r3) --  (U9r6) --  (U9r0) --  (U9r4) --  (U9r8) --  (U9r3) --  (U9r7) --  (U9r2) --  (U9r6) --  (U9r0) --  (U9r5) --  (U9r1) --  (U9r6) --  (U9r2) --  (U9r7) --  (U9r3) --  (U9r8) --  (U9r4) --  (U9r0); 
      \draw (U9r1) -- (U9r4) -- (U9r7) -- (U9r1);
      \draw (U9r2) -- (U9r5) -- (U9r8) -- (U9r2);
    
  \end{tikzpicture}\\
  \caption{The graph $F$.}
  \label{fig:test2}
\end{minipage}
\end{figure}

The main result of this section is the following:  

\begin{theorem}
\label{mytheorem} For any positive integers $m,n$ ($m\leq n$), there exists a connected graph $G$ of bounded maximum degree such that $\mathrm{def}_{c}(G)=m$ and $\mathrm{def}(G)=n$.
\end{theorem}
\begin{proof}
We shall construct a connected graph $G_{m,n}$ satisfying the conditions of
 the theorem. Let $H$ be the graph shown in Figure \ref{fig:test1} 
and $F$ be the graph shown in Figure \ref{fig:test2}. 
In \cite{GiaroKubaleMalaf1,PetrosHrant}, it was proved that 
$\mathrm{def}(K_{5})=2$ and $\mathrm{def}(K_{9}-e)=3$. 
This implies that $\mathrm{def}(H)\geq 1$ and $\mathrm{def}(F)\geq 1$. 
Let us consider the graph $G_{m,n}$ shown in Figure \ref{K9-minus-edge-K5-star};
this graph contains $m$ copies of $H$ and $n-m$ copies of $F$. 
Clearly, $G_{m,n}$ is a connected graph and $\Delta(G_{m,n})=12$ 
for any $m,n\in \mathbb{N}$. 
Since the graph $G_{m,n}$ contains $m$ copies of $H$ and $n-m$ 
copies of $F$, $\mathrm{def}(G_{m,n})\geq n$.
On the other hand, the coloring in
Figure \ref{K9-minus-edge-K5-star-def} implies
that $\mathrm{def}(G_{m,n})\leq n$.
Thus $\mathrm{def}(G_{m,n})=n$ for any $m,n\in \mathbb{N}$.

Let us now show that $\mathrm{def}_{c}(G_{m,n})=m$ for any $m,n\in \mathbb{N}$. 
We first note that the graphs $H$ and $F$ are cyclically interval colorable. Moreover, it can be seen from Figure \ref{K9-minus-edge-K5-star-cyclic-def} 
that $\mathrm{def}_{c}(G_{m,n})\leq m$. 

It remains to prove that $\mathrm{def}_{c}(G_{m,n})\geq m$.
Let $J$ be the graph shown in Figure \ref{fig}.
In \cite{PetMkhitaryan}, it was proved that $J\notin \mathfrak{N}_{c}$, 
thus $\mathrm{def}_{c}(J)\geq 1$. 
Let $\alpha$ be a proper $t$-edge-coloring of $G_{m,n}$ with a 
minimum cyclic deficiency, that is 
$\mathrm{def}_{c}(G_{m,n},\alpha)=\mathrm{def}_{c}(G_{m,n})$. 
Suppose, for a contradiction, that $\mathrm{def}_{c}(G_{m,n})< m$. 
Since $G_{m,n}$ contains $m$ copies of $J$,
this implies that there exists a copy $J^{\prime}$ of the graph $J$ in $G_{m,n}$ such that for every $v\in V(J^{\prime})$, the set $S_{J^{\prime}}\left(v,\alpha \right)$ is a cyclic interval modulo $t$. This implies that $J$ has a cyclic
interval coloring, a contradiction.
%
%
$\square$
\end{proof}

\begin{figure}
  \begin{center}
  \input{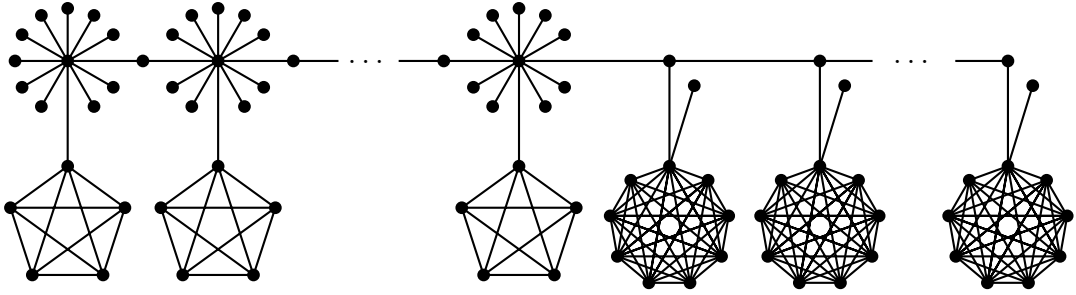}
  \end{center}
  \caption{The graph $G_{m,n}$.}
  \label{K9-minus-edge-K5-star}
\end{figure}

\begin{figure}
  \begin{center}
  \input{graph_deficiency.tikz}
  \end{center}
  \caption{$\mathrm{def}(G_{m,n})\leq n$.}
  \label{K9-minus-edge-K5-star-def}
\end{figure}

\begin{figure}
  \begin{center}
  \input{graph_cyclic_deficiency.tikz}
  \end{center}
  \caption{$\mathrm{def}_{c}(G_{m,n})\leq m$.}
  \label{K9-minus-edge-K5-star-cyclic-def}
\end{figure}

The above theorem has some immediate consequences.

\begin{corollary}
\label{corollary1} For any $n\in \mathbb{N}$, there exists a connected graph 
$G$ of bounded maximum degree such that $\mathrm{def}_{c}(G)\geq n$.
\end{corollary}

\begin{corollary}
\label{corollary2} For any $n\in \mathbb{N}$, there exists a connected graph $G$ of bounded maximum degree such that  $\mathrm{def}(G)-\mathrm{def}_{c}(G)\geq n$.
\end{corollary}

\bigskip

 \section{Graphs with large cyclic deficiency}\

In this section we shall describe three methods for constructing classes of
graphs with large cyclic deficiency. 
Our methods generalize previous constructions of graphs without interval colorings
by Hertz, Giaro et al. and Erd\H os (see e.g. 
\cite{GiaroKubaleMalaf1,JensenToft,GiaroKubaleMalaf2}) and give new classes of graphs with 
large deficiencies.

\subsection{Constructions using subdivisions}\bigskip

Let $G$ be a graph with $V(G)=\{v_{1},\ldots,v_{n}\}$. Define the graphs
$S(G)$ and $\widehat{G}$ as follows:
\begin{center}
$V(S(G))=\{v_{1},\ldots,v_{n}\}\cup \{w_{ij}:v_{i}v_{j}\in E(G)\}$,
\end{center}
\begin{center}
$E(S(G))=\{v_{i}w_{ij},v_{j}w_{ij}:v_{i}v_{j}\in E(G)\}$,
\end{center}
\begin{center}
$V(\widehat{G})=V(S(G))\cup \{u\}$, $u\notin V(S(G))$,
$E(\widehat{G})=E(S(G))\cup \{uw_{ij}:v_{i}v_{j}\in E(G)\}$.
\end{center}

In other words, $S(G)$ is the graph obtained by subdividing every
edge of $G$, and $\widehat{G}$ is the graph obtained from $S(G)$ by
connecting every inserted vertex to a new vertex $u$. Clearly,
$S(G)$ and $\widehat{G}$ are bipartite graphs.

We will use the following two results.

\begin{theorem}
\label{mytheorem5} \cite{DanDevPra} If $G$ is a connected graph with at least two vertices, then 
$\mathrm{diam}(S(G))\leq 2\mathrm{diam}(G)+2$.
\end{theorem}

\begin{theorem}
\label{mytheorem6} \cite{PetMkhitaryan} If $G$ is a connected bipartite graph and $G\in
\mathfrak{N}_{c}$, then $
W_{c}(G)\leq 1+2 \mathrm{diam}(G)\left(\Delta(G)-1\right)$.
\end{theorem}

The following is the main result of this subsection.

\begin{theorem}
\label{mytheorem7} If $G$ is a connected graph with at least two vertices, then
\begin{center}
$\mathrm{def}_{c}(\widehat{G})\geq \frac{\left(\vert E(G)\vert -1\right)}
{4\left(\mathrm{diam}(G)+2\right)}-\Delta(G)+1$.
\end{center}
\end{theorem}
\begin{proof}
Let $\alpha$ be a proper $t$-edge-coloring of $\widehat{G}$ with the minimum cyclic deficiency, that is,  $\mathrm{def}_{c}(\widehat{G},\alpha)=\mathrm{def}_{c}(\widehat{G})$. Clearly, $t\geq \Delta(\widehat{G})=\vert E(G)\vert$.

Define an auxiliary graph $\widehat{G}^{\prime}$ as follows: for each vertex $v\in V(\widehat{G})$ with $\mathrm{def}_{c}(v,\alpha)>0$, we attach $\mathrm{def}_{c}(v,\alpha)$ pendant edges at vertex $v$. Clearly, $\vert V(\widehat{G}^{\prime})\vert = \vert V(\widehat{G})\vert +\mathrm{def}_{c}(\widehat{G})$. 
Next, 
define the graph $\widehat{G}^{+}$ from $\widehat{G}^{\prime}$ as follows: 
we remove the vertex $u$ from $\widehat{G}^{\prime}$ and all isolated vertices from $\widehat{G}^{\prime}-u$ (if such vertices exist) 
and add a new vertex $u_{ij}$ for each inserted vertex $w_{ij}$; then we add new edges $u_{ij}w_{ij}$ ($v_{i}v_{j}\in E(G)$). The graph $\widehat{G}^{+}$ 
is connected and bipartite, and has 
edge set $E(\widehat{G}^{\prime}-u)\cup \{u_{ij}w_{ij}:v_{i}v_{j}\in E(G)\}$. Moreover, by Theorem \ref{mytheorem5}, we obtain that 
$\mathrm{diam}(\widehat{G}^{+})\leq 2\mathrm{diam}(G)+4$.

We extend the proper $t$-edge-coloring $\alpha$ of $\widehat{G}$ to a proper $t$-edge-coloring $\beta$ of $\widehat{G}^{+}$ as follows: for each vertex $v\in V(\widehat{G}^{\prime}-u)$ with $\mathrm{def}_{c}(v,\alpha)>0$, we color the attached edges incident to $v$ using $\mathrm{def}_{c}(v,\alpha)$ distinct colors to obtain cyclic interval modulo $t$, and for each inserted vertex $w_{ij}$, we color the edge $u_{ij}w_{ij}$ with color $\alpha(uw_{ij})$. 
By the definition 
of $\beta$
and the construction of $\widehat{G}^{+}$, we obtain that 
$\beta$ is a cyclic interval $t$-coloring. By Theorem \ref{mytheorem6}, we have
\begin{center}
$\vert E(G)\vert\leq t\leq 1+2 \mathrm{diam}(\widehat{G}^{+})\left(\Delta(\widehat{G}^{+})-1\right)\leq 1+4\left(\mathrm{diam}(G)+2\right)\left(\Delta(G)+\mathrm{def}_{c}(\widehat{G})-1\right)$.
\end{center}

Hence
\begin{center}
$\mathrm{def}_{c}(\widehat{G})\geq \frac{\left(\vert E(G)\vert -1\right)}{4\left(\mathrm{diam}(G)+2\right)}-\Delta(G)+1$.  ~$\square$
\end{center}
\end{proof}

Using Theorem \ref{mytheorem7}, we can generate infinite
families of graphs with large cyclic deficiency. Let us consider
some examples.

For the complete graph $K_{n}$ we have 
$\mathrm{diam}(K_{n}) = 1$ and $\Delta(K_{n})=n-1$;
so, by Theorem \ref{mytheorem7}, 
$\mathrm{def}_{c}(\widehat{K}_{n})\geq \frac{(n^{2}-n-2)}{24}-n+2$.

Next, for the complete bipartite graph $K_{m,n}$ it holds that 
$\mathrm{diam}(K_{m,n})\leq 2$ and $\Delta(K_{m,n})=\max\{m,n\}$; thus
by Theorem \ref{mytheorem7}, 
$\mathrm{def}_{c}(\widehat{K}_{m,n})\geq \frac{(m n-1)}{8}-\max\{m,n\}+1$.

Finally, 
for the hypercube $Q_{n}$, $\vert E(Q_{n})\vert =n2^{n-1}$ and 
$\mathrm{diam}(Q_{n})=\Delta(Q_{n})=n$. Using Theorem \ref{mytheorem7} 
we deduce that $\mathrm{def}_{c}(\widehat{Q}_{n})\geq \frac{\left(n2^{n-1}-1\right)}{4(n+2)}-n+1$. 

\subsection{Constructions using trees}\bigskip

Our next construction uses
techniques first described in \cite{PetrosKhacha} and generalizes the family of so-called Hertz graphs
first described in \cite{GiaroKubaleMalaf1}. 

%
Let $T$ be a tree and $V(T)=\{v_{1},\ldots,v_{n}\}$, $n\geq 2$.
Also, let $F(T)=\{v:v\in V(T)\wedge d_{T}(v)=1\}$. For a path
$P(v_{i},v_{j})$, define $LP(v_{i},v_{j})$ and $M(T)$ as follows:
\begin{center}
$LP(v_{i},v_{j})=\vert EP(v_{i},v_{j})\vert +\vert\left\{vw:vw\in
E(T), v\in VP(v_{i},v_{j}), w\notin VP(v_{i},v_{j})\right\}\vert$,
\end{center}

\begin{center}
$M(T)={\max }_{1\leq i<j\leq n}LP(v_{i},v_{j})$.
\end{center}

Now let us define the graph $\widetilde{T}$ as follows:
\begin{center}
$V(\widetilde{T})=V(T)\cup \{u\}$, $u\notin V(T)$,
$E(\widetilde{T})=E(T)\cup \{uv:v\in F(T)\}$.
\end{center}

Clearly, $\widetilde{T}$ is a connected graph with
$\Delta(\widetilde{T})=\vert F(T)\vert$. Moreover, if $T$ is a tree
in which the distance between any two pendant vertices is even, then
$\widetilde{T}$ is a connected bipartite graph.\\

In \cite{Kam_Trees}, Kamalian proved the following result.

\begin{theorem}
\label{mytheorem1} If $T$ is a tree, then $T\in \mathfrak{N}_{c}$ and $W_{c}(T)=M(T)$.
\end{theorem}

Using this theorem we prove the following:

\begin{theorem}
\label{mytheorem2} If $T$ is a tree, then $\mathrm{def}_{c}(\widetilde{T})\geq \vert F(T)\vert -M(T)-2$.
\end{theorem}
\begin{proof}
Let $F(T)=\{v_{1},\ldots,v_{p}\}$ and $\alpha$ be a proper $t$-edge-coloring of $\widetilde{T}$ with the minimum cyclic deficiency, that is $\mathrm{def}_{c}(\widetilde{T},\alpha)=\mathrm{def}_{c}(\widetilde{T})$. Clearly, $t\geq \vert F(T)\vert$.

Define an auxiliary graph $\widetilde{T}^{\prime}$ as follows: for each vertex $v\in V(\widetilde{T})$ with $\mathrm{def}_{c}(v,\alpha)>0$, we attach $\mathrm{def}_{c}(v,\alpha)$ pendant edges at vertex $v$. Clearly, $\vert V(\widetilde{T}^{\prime})\vert = \vert V(\widetilde{T})\vert +\mathrm{def}_{c}(\widetilde{T})$. Next, for a graph $\widetilde{T}^{\prime}$ and $F(T)=\{v_{1},\ldots,v_{p}\}$, define the graph $T^{+}$ as follows: we remove the vertex $u$ from $\widetilde{T}^{\prime}$ and all isolated vertices from $\widetilde{T}^{\prime}-u$ (if exist) and add new vertices $u_{1},\ldots,u_{p}$, then we add new edges $u_{1}v_{1},\ldots,u_{p}v_{p}$. Clearly, $T^{+}$ is a tree with edge set $E(\widetilde{T}^{\prime}-u)\cup \{u_{1}v_{1},\ldots,u_{p}v_{p}\}$. Moreover, it is easy to see that $M(T^{+})\leq M(T)+2+\mathrm{def}_{c}(\widetilde{T})$.

We now extend a proper $t$-edge-coloring $\alpha$ of $\widetilde{T}$ to a proper $t$-edge-coloring $\beta$ of $T^{+}$ as follows: for each vertex $v\in V(\widetilde{T}^{\prime}-u)$ with $\mathrm{def}_{c}(v,\alpha)>0$, we color the attached edges incident to $v$ using $\mathrm{def}_{c}(v,\alpha)$ distinct colors to obtain cyclic interval modulo $t$, and for each $1\leq i\leq p$, we color the edge $u_{i}v_{i}$ with color $\alpha(uv_{i})$. By the definition of $\beta$
and the construction of $T^{+}$, we obtain that $T^{+}$ has a cyclic interval $t$-coloring. Since $W_{c}(T)=M(T)$ (by Theorem \ref{mytheorem1}), we have
\begin{center}
$\vert F(T)\vert\leq t\leq W_{c}(T^{+})=M(T^{+})\leq M(T)+2+\mathrm{def}_{c}(\widetilde{T})$.
\end{center}

Hence
\begin{center}
$\mathrm{def}_{c}(\widetilde{T})\geq \vert F(T)\vert -M(T)-2$.
\end{center} ~$\square$
\end{proof}

We note the following corollaries.

\begin{corollary}
\label{mycorollary1} If $T$ is a tree, then $\mathrm{def}(\widetilde{T})\geq \vert F(T)\vert -M(T)-2$.
\end{corollary}

\begin{corollary}
\label{mycorollary2} If $T$ is a tree in which the distance
between any two pendant vertices is even, then the graph $\widetilde{T}$
is bipartite, and $\mathrm{def}_{c}(\widetilde{T})\geq \vert F(T)\vert -M(T)-2$.
\end{corollary}

Our constructions by trees generalize the so-called
Hertz's graphs $H_{p,q}$ first described in
\cite{GiaroKubaleMalaf1}. Hertz's graphs
are known to have a high deficiency so let us specifically 
consider the cyclic deficiency of such graphs. 

In \cite{GiaroKubaleMalaf1} the Hertz's graph 
$H_{p,q}$ $(p,q\geq 2)$ was defined as
follows:
\begin{center}
$V(H_{p,q})=\left\{a,b_{1},b_{2},\ldots,b_{p},d\right\}\cup
\left\{c_{j}^{(i)}:1\leq i\leq p, 1\leq j\leq q\right\}$ and
\end{center}
\begin{center}
$E(H_{p,q})=E_{1}\cup E_{2}\cup E_{3}$,
\end{center}
where
\begin{center}
$E_{1}=\left\{ab_{i}:1\leq i\leq p\right\}$,
$E_{2}=\left\{b_{i}c_{j}^{(i)}:1\leq i\leq p, 1\leq j\leq
q\right\}$, $E_{3}=\left\{c_{j}^{(i)}d:1\leq i\leq p, 1\leq j\leq q
\right\}$.
\end{center}

The graph $H_{p,q}$ is bipartite with maximum degree $\Delta (H_{p,q})=pq$ and
$\vert V(H_{p,q})\vert=pq+p+2$. For Hertz's graphs, Giaro, Kubale and
Malafiejski proved the following theorem.

\begin{theorem}
\label{mytheorem3} \cite{GiaroKubaleMalaf1} For any positive integers $p\geq 4$, $q\geq 3$,
\begin{center}
$\mathrm{def}(H_{p,q})=pq-p-2q-2$.
\end{center}
\end{theorem}

Note that this result was recently 
generalized by Borowiecka-Olszewska et al. \cite{B-OD-B}.
Using Theorems \ref{mytheorem2} and \ref{mytheorem3} we show that the following result holds.

\begin{theorem}
\label{mytheorem4} For any positive integers $p\geq 4$, $q\geq 3$, we have
\begin{center}
$\mathrm{def}_{c}(H_{p,q})=\mathrm{def}(H_{p,q})=pq-p-2q-2$.
\end{center}
\end{theorem}
\begin{proof} 
%
Let us consider the tree $T=H_{p,q}-d$. Since $M(T)=p+2q$, $\vert F(T)\vert=pq$ and taking into account that the graph $H_{p,q}$ is isomorphic to $\widetilde{T}$, by Theorem \ref{mytheorem2}, we obtain that $\mathrm{def}_{c}(H_{p,q})\geq pq-p-2q-2$. On the other hand, by Theorem \ref{mytheorem3}, we have $\mathrm{def}_{c}(H_{p,q})\leq \mathrm{def}(H_{p,q})=pq-p-2q-2$. Hence, $\mathrm{def}_{c}(H_{p,q})=\mathrm{def}(H_{p,q})=pq-p-2q-2$ for any $p\geq 4$, $q\geq 3$.
 ~$\square$
\end{proof}

\begin{figure}[h]
\begin{center}
\includegraphics[width=20pc]{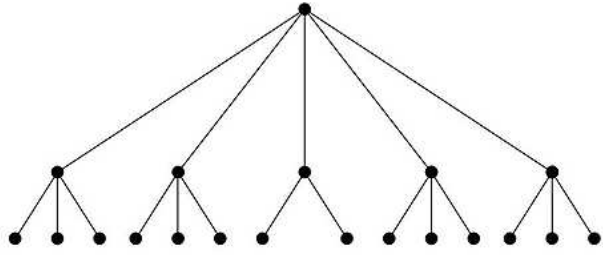}\\
\caption{The tree $T$.}\label{fig3}
\end{center}
\end{figure}

Finally, let us remark
that the above technique can be used for constructing the smallest, in terms
of maximum degree, currently known example of a bipartite graph with no cyclic
interval coloring (cf. \cite{AsratianCasselgrenPetrosyan}).
To this end, consider the tree $T$ shown in Figure \ref{fig3}. 
Since $M(T)=11$ and $\vert F(T)\vert=14$, the bipartite graph $\widetilde{T}$
with $\vert V(\widetilde{T})\vert=21$ and
$\Delta(\widetilde{T})=14$ has no cyclic interval coloring 
(See Figure \ref{fig4}). 

\begin{figure}[h]
\begin{center}
\includegraphics[width=25pc]{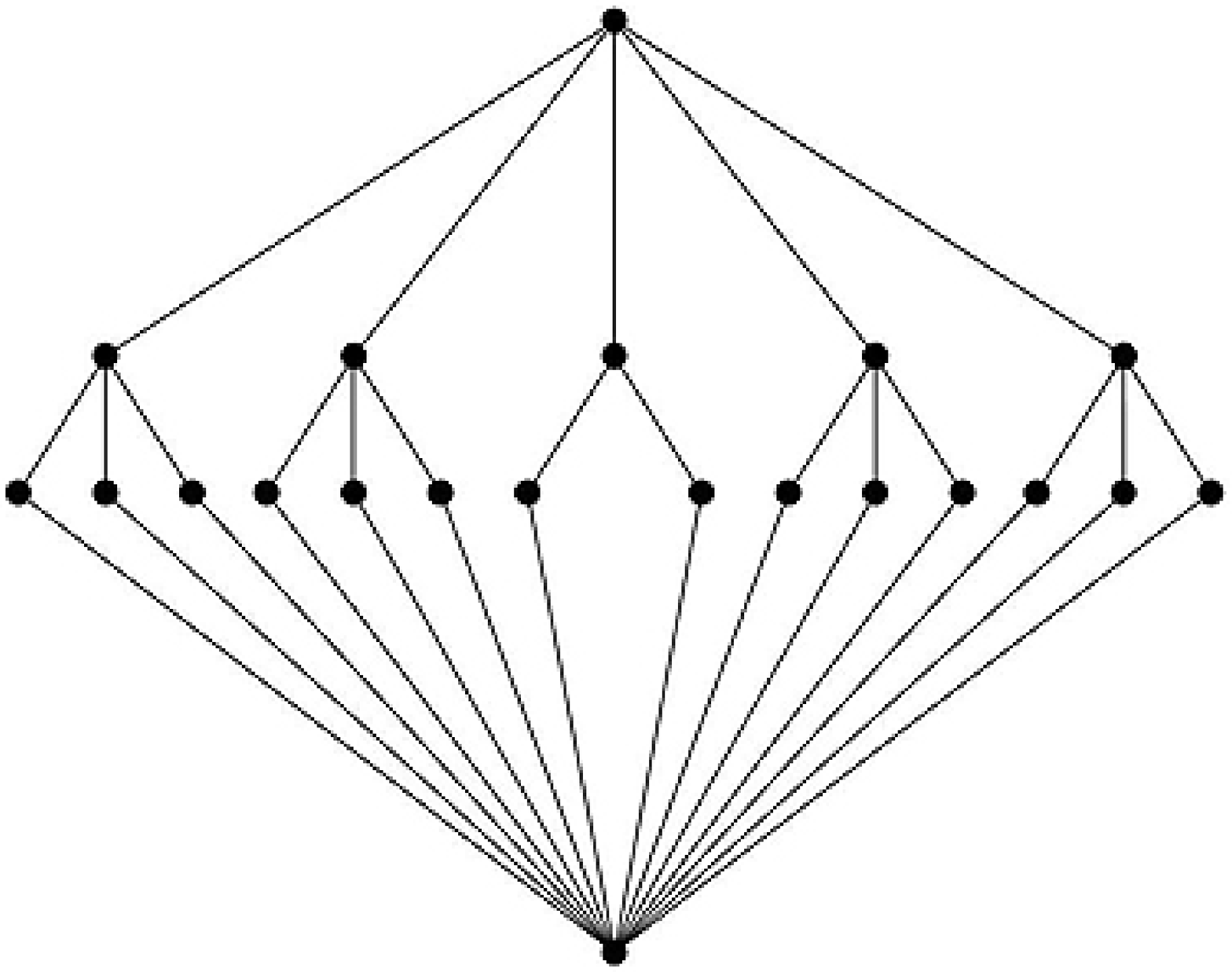}\\
\caption{The graph $\widetilde{T}$.}\label{fig4}
\end{center}
\end{figure}


\subsection{Constructions using finite projective planes}\bigskip

In the last part of this section we use finite projective planes (see Definition \ref{def:project}) for constructing
bipartite graphs with large cyclic deficiency. 
This family of graphs was first described in \cite{PetrosKhacha}.

Let $\pi(n)$ be a finite projective
plane of order $n\geq 2$, $\{1,2,\ldots,n^{2}+n+1\}$ be the set of points and $L$  the set
of lines of $\pi(n)$.  For a sequence of $n^2+n+1$ integers 
$r_{1}, r_2,\ldots, r_{n^{2}+n+1}\in \mathbb{N}$, define the graph
$Erd(r_{1},\ldots,r_{n^{2}+n+1})$  as follows:

\begin{center}
$V(Erd(r_{1},\ldots,r_{n^{2}+n+1}))=\{u\}\cup\{1,\ldots,n^{2}+n+1\}$\\
$\cup \left\{v^{(l_{i})}_{1},\ldots,v^{(l_{i})}_{r_{i}}:l_{i}\in L,1\leq i\leq n^{2}+n+1\right\}$,\\
\end{center}
\begin{center}
$E(Erd(r_{1},\ldots,r_{n^{2}+n+1}))=\left\{uv^{(l_{i})}_{1},\ldots,uv^{(l_{i})}_{r_{i}}:l_{i}\in
L, 1\leq i\leq
n^{2}+n+1\right\}\cup$\\
$\bigcup_{i=1}^{n^{2}+n+1}\left\{v^{(l_{i})}_{1}k,\ldots,v^{(l_{i})}_{r_{i}}k:l_{i}\in
L,k\in l_{i},1\leq k\leq n^{2}+n+1\right\}$.
\end{center}

Clearly, $Erd(r_{1},r_2,\ldots,r_{n^{2}+n+1})$ is a connected bipartite
graph where the number of vertices  is $n^2+n+2+\underset{i=1}{\overset{n^{2}+n+1}{\sum
}}r_{i}$ \  \
and the maximum degree is $\underset{i=1}{\overset{n^{2}+n+1}{\sum}}r_{i}$.

%
 Note that the above  graph with parameters $n=3$ and $r_1=r_2=\dots=r_{13}=1$ (see Figure \ref{fig5})
 was described in 1991 by Erd\H{o}s \cite{JensenToft}.

\begin{figure}[h]
\begin{center}
\includegraphics[width=30pc]{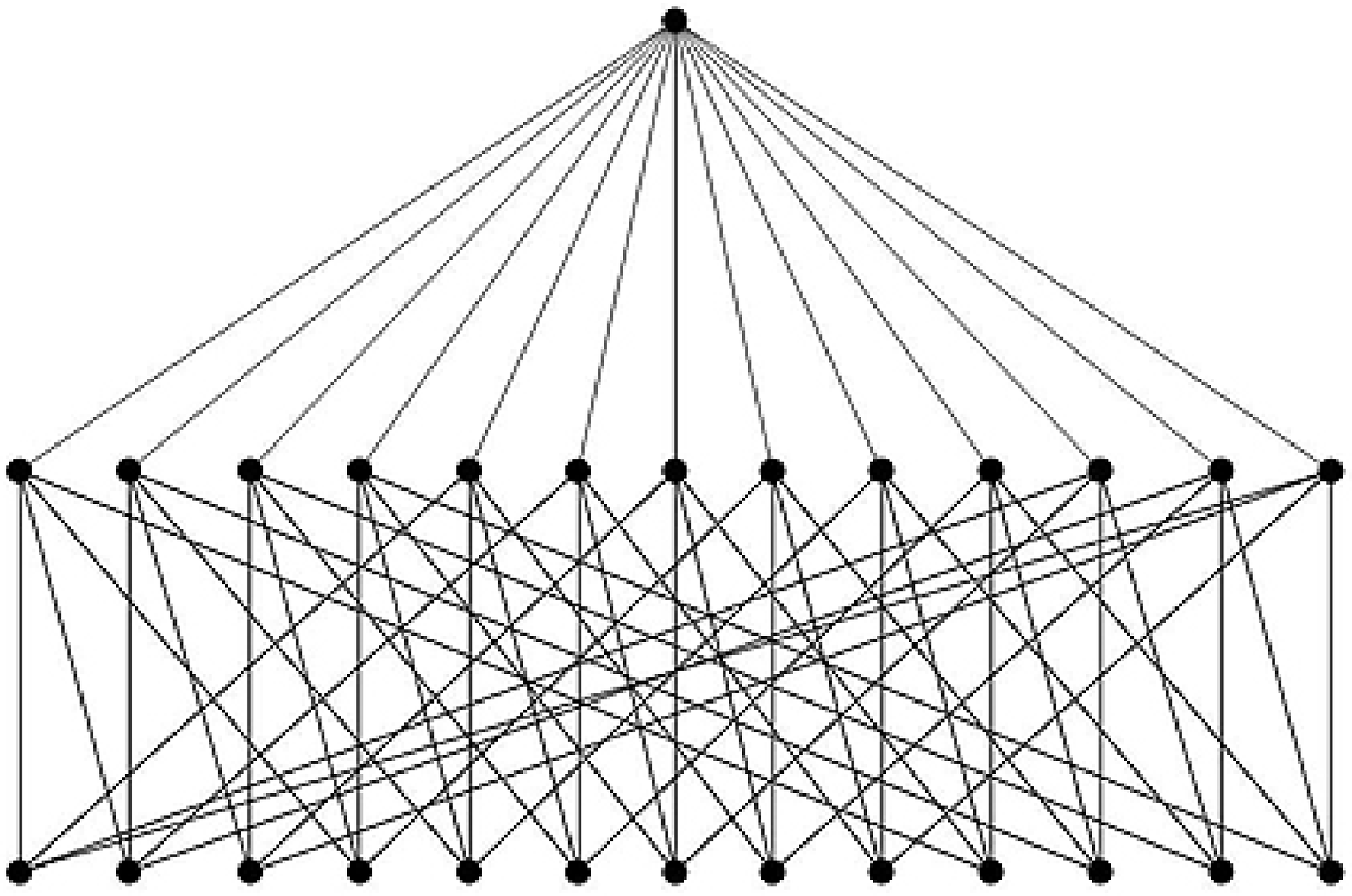}\\
\caption{Erd\H{o}s' graph.}\label{fig5}
\end{center}
\end{figure}

\begin{theorem}
\label{mytheorem2.2.1} 
For any sequence $r_{1},r_2,\ldots,r_{n^{2}+n+1}\in \mathbb{N}$ with $r_{1}\geq r_2\geq \ldots\geq r_{n^{2}+n+1}$, 
we have 
\begin{center}
$\mathrm{def}_{c}(Erd(r_{1},r_2,\ldots,r_{n^{2}+n+1}))\geq \dfrac{1}{10}\left(\underset{i=n+2}{\overset{n^{2}+n+1}{\sum
}}r_{i}-9\underset{i=1}{\overset{n+1}{\sum
}}r_{i}+9\right)$.
\end{center}
\end{theorem}
\begin{proof}
Let $\alpha$ be a proper $t$-edge-coloring of $Erd(r_{1},\ldots,r_{n^{2}+n+1})$ with the minimum cyclic deficiency, that is,
$$\mathrm{def}_{c}(Erd(r_{1},r_2,\ldots,r_{n^{2}+n+1}),\alpha)=\mathrm{def}_{c}(Erd(r_{1},r_2,\ldots,r_{n^{2}+n+1})).$$ Clearly, 
$t\geq \underset{i=1}{\overset{n^{2}+n+1}{\sum
}}r_{i}$, because the maximum degree of $Erd(r_{1},\ldots,r_{n^{2}+n+1})$ is $\underset{i=1}{\overset{n^{2}+n+1}{\sum
}}r_{i}$.

Define an auxiliary graph $\widehat{Erd}(\overline{r})$ as follows:
for each vertex $v\in V(Erd(r_{1},\ldots,r_{n^{2}+n+1}))$ with 
$\mathrm{def}_{c}(v,\alpha)>0$, we attach $\mathrm{def}_{c}(v,\alpha)$ pendant edges at vertex $v$. 
Clearly, $\vert V(\widehat{Erd}(\overline{r}))\vert = \vert V(Erd(r_{1},\ldots,r_{n^{2}+n+1}))\vert +\mathrm{def}_{c}(Erd(r_{1},\ldots,r_{n^{2}+n+1}))$. 
Next, 
from $\widehat{Erd}(\overline{r})$ we define the graph 
$\widehat{Erd}^{+}(\overline{r})$ as follows: 
we remove the vertex $u$ from $\widehat{Erd}(\overline{r})$ and all 
isolated vertices from $\widehat{Erd}(\overline{r})-u$ 
(if such vertices exist)
and add new vertices $w^{(l_{i})}_{1},\ldots,w^{(l_{i})}_{r_{i}}$ for $l_{i}\in
L, 1\leq i\leq n^{2}+n+1$; then we add new edges $$v^{(l_{i})}_{1}w^{(l_{i})}_{1},\ldots,v^{(l_{i})}_{r_{i}}w^{(l_{i})}_{r_{i}}, \quad l_{i}\in
L, 1\leq i\leq n^{2}+n+1.$$ 
The graph $\widehat{Erd}^{+}(\overline{r})$ is
connected bipartite graph 
and has diameter at most $5$.

We now extend the proper $t$-edge-coloring 
$\alpha$ of $Erd(r_{1},\ldots,r_{n^{2}+n+1})$ 
to a proper $t$-edge-coloring $\beta$ of 
$\widehat{Erd}^{+}(\overline{r})$ as follows: for each vertex 
$v\in V(\widehat{Erd}(\overline{r})-u)$ with 
$\mathrm{def}_{c}(v,\alpha)>0$, we color the attached 
edges incident to $v$ using $\mathrm{def}_{c}(v,\alpha)$ distinct colors 
to obtain a cyclic interval modulo $t$ at $v$, and we color 
the edges 
$v^{(l_{i})}_{1}w^{(l_{i})}_{1},\ldots,v^{(l_{i})}_{r_{i}}w^{(l_{i})}_{r_{i}}$ 
with colors $\alpha\left(uv^{(l_{i})}_{1}\right),\ldots,
\alpha\left(uv^{(l_{i})}_{r_{i}}\right)$ ($l_{i}\in
L, 1\leq i\leq n^{2}+n+1$), respectively. By the definition of $\beta$
and the construction of $\widehat{Erd}^{+}(\overline{r})$, we obtain 
a cyclic interval $t$-coloring of
$\widehat{Erd}^{+}(\overline{r})$. 
Now, by Theorem \ref{mytheorem6}, we have

\begin{align*}
\underset{i=1}{\overset{n^{2}+n+1}
{\sum}}r_{i}  \leq t &\leq 1+2 \, \mathrm{diam}(\widehat{Erd}^{+}(\overline{r}))
\left(\Delta(\widehat{Erd}^{+}(\overline{r}))-1\right) \\
& \leq 1+10\left(\Delta(Erd(r_{1},\ldots,r_{n^{2}+n+1}))+
\mathrm{def}_{c}(Erd(r_{1},\ldots,r_{n^{2}+n+1}))-1\right) \\
& \leq 10\underset{i=1}{\overset{n+1}{\sum
}}r_{i}+10 \, \mathrm{def}_{c}(Erd(r_{1},\ldots,r_{n^{2}+n+1}))-9.
\end{align*}

Hence
\begin{center}
$\mathrm{def}_{c}(Erd(r_{1},\ldots,r_{n^{2}+n+1})) \geq \dfrac{1}{10}
\left(\underset{i=n+2}{\overset{n^{2}+n+1}
{\sum}}r_{i}-9\underset{i=1}{\overset{n+1}{\sum
}}r_{i}+9\right)$.  ~$\square$
\end{center} 
\end{proof}

Using Theorem \ref{mytheorem2.2.1} we can generate infinite families of graphs
with large cyclic deficiency. For example, if 
$r_1 = r_2 = \dots = r_{n^2+n+1} = k$, where  $k$ is some constant, then 
$\mathrm{def}_c (Erd(r_{1},\ldots,r_{n^{2}+n+1})) \geq 
\frac{1}{10} (n^2 k - 9 k (n+1) + 9)$.

\bigskip

\section{Upper bounds on $\text{def}_{c}(G)$}\

In this section we present some upper bounds on the cyclic deficiency
of graphs. 
 Theorem \ref{th:Vizing} implies that if every vertex of a 
graph $G$ has degree 1 or $\Delta(G)$, 
then any proper edge coloring of $G$ with $\chi'(G)$
	colors is a cyclic interval coloring of $G$. 
	This implies that the graph obtained from $G$ by
	attaching $\Delta(G)-d_G(v)$ pendant edges
	at each vertex $v$ with $1<d_G(v)<\Delta(G)$ is cyclically interval colorable.
Therefore, for any graph $G$ \bigskip

\centerline{$\mathrm{def}_{c}(G)\leq 
\sum\limits_{\begin{subarray}{c} v\in V(G),\\2\leq d_{G}(v)\leq \Delta(G)-1 \end{subarray}}
(\Delta(G)-d_{G}(v)).$} \medskip

%
%
%


In \cite{AsratianCasselgrenPetrosyan} we proved that all bipartite
graphs with maximum degree $4$ admit cyclic interval colorings.
For bipartite graphs with maximum degree $\Delta(G)\geq 5$
the above upper bound on the cyclic deficiency can be slightly improved:

\begin{proposition}
\label{prop:upperbound}
	If $G$ is a bipartite graph with maximum degree $\Delta(G) \geq 5$,
	then 
	
\begin{center}
$\mathrm{def}_{c}(G)\leq 
\left\{
\begin{tabular}{ll}
$\sum\limits_{\begin{subarray}{c} v\in V(G),\\3\leq d_{G}(v)\leq \Delta(G)-3 \end{subarray}}
(\Delta(G)-2-d_{G}(v))$, & if $\Delta(G)$ is even,\\
$\sum\limits_{\begin{subarray}{c} v\in V(G),\\3\leq d_{G}(v)\leq \Delta(G)-2 \end{subarray}}
(\Delta(G)-1-d_{G}(v))$, & if $\Delta(G)$ is odd.\\
\end{tabular}%
\right.$
\end{center}

\end{proposition}
\begin{proof} 
In the proof we follow the idea from the proof of Theorem 1 in
\cite{AsratianCasselgrenPetrosyan}. 

First we consider the case when $\Delta(G)$ is even.
In this case we construct a new multigraph $G^{\star}$ as follows:
first we take two isomorphic copies $G_{1}$ and $G_{2}$ of the graph
$G$ and join by an edge every vertex with an odd vertex degree in
$G_{1}$ with its copy in $G_{2}$; then for each vertex $u \in
V(G_{1}) \cup V(G_{2})$ of degree $2k$, we add 
$\frac{\Delta(G)}{2}-k$ loops at $u$ $\left(1\leq k< \frac{\Delta(G)}{2}\right)$. Clearly, $G^{\star}$ is a $\Delta(G)$-regular
multigraph. By Petersen's theorem, $G^{\star}$ can be represented as a union
of edge-disjoint $2$-factors $F_{1},\ldots,F_{\frac{\Delta(G)}{2}}$. By removing all
loops from $2$-factors $F_{1},\ldots,F_{\frac{\Delta(G)}{2}}$ of $G^{\star}$, we
obtain that the resulting graph $G^{\prime}$ is a union of edge-disjoint
even subgraphs $F^{\prime}_{1},\ldots,F^{\prime}_{\frac{\Delta(G)}{2}}$. Since
$G^{\prime}$ is bipartite, for each $i$ $\left(1\leq i\leq \frac{\Delta(G)}{2}\right)$,
$F^{\prime}_{i}$ is a collection of even cycles in $G^{\prime}$, and
we can color the edges of $F^{\prime}_{i}$ alternately with colors
$2i-1$ and $2i$.  The resulting coloring  $\alpha$
is a proper edge coloring of
$G^{\prime}$ with colors $1,\ldots,\Delta(G)$. Since for each vertex $u\in
V(G^{\prime})$ with $d_{G^{\prime}}(u)=2k$ $\left(1\leq k\leq \frac{\Delta(G)}{2}\right)$, there are $k$
even subgraphs $F^{\prime}_{i_{1}},F^{\prime}_{i_{2}},\ldots,F^{\prime}_{i_{k}}$ such that
$d_{F^{\prime}_{i_{1}}}(u)=d_{F^{\prime}_{i_{2}}}(u)=\cdots=d_{F^{\prime}_{i_{k}}}(u)=2$, we obtain that
$S_{G^{\prime}}(u,\alpha)=\{2i_{1}-1,2i_{1},2i_{2}-1,2i_{2},\ldots,2i_{k}-1,2i_{k}\}$.
This implies that for the restriction $\beta$ of this proper edge
coloring to the edges of the graph $G$, 
we have $$\mathrm{def}_{c}(G,\beta)\leq \sum\limits_{\begin{subarray}{c} v\in V(G),\\3\leq d_{G}(v)\leq \Delta(G)-3 \end{subarray}} (\Delta(G)-2-d_{G}(v)).$$ 

Next we consider the case when $\Delta(G)$ is odd.
In this case we construct a new graph $G^{\star\star}$ as follows:
we take two isomorphic copies of the graph $G$ and join by an edge 
every vertex of
maximum degree with its copy. 
Clearly, $G^{\star\star}$ is a bipartite graph with $\Delta(G^{\star\star})=\Delta(G)+1$. Now the result follows from the first part of the proof. 
~$\square$
\end{proof}

In the following, for a graph $G$, we denote by $V_{k}(G)$ (or just $V_k$) 
the set of vertices of degree $k$ in $G$.

\begin{corollary}
\label{mycorollary6}
    If $G$ is a bipartite graph with $\Delta(G)\leq 6$,
	then $\mathrm{def}_{c}(G) \leq \vert V_{3}\vert$.
\end{corollary}

\begin{corollary}
\label{mycorollary7}
	If $G$ is a bipartite graph where, for some $r \geq 2$, all vertex degrees are
	in the set $\{2r-3,2r-2, 2r-1,2r\}$,
	then $\mathrm{def}_{c}(G) \leq \vert V_{2r-3}\vert$.
\end{corollary}

In the preceding section we proved
that there are families of bipartite graphs $G_{n}$ such that $\lim_{n\rightarrow \infty} \frac{\mathrm{def}_{c}(G_{n})}{\vert V(G_{n})\vert} = 1$. 
%
On the other hand, we do not know graphs $G$ with 
$\mathrm{def}_{c}(G)>\vert V(G)\vert$.  
Hence, the following conjecture seems natural:

\begin{conjecture}
\label{conj:upperbound}
For any graph $G$, $\mathrm{def}_{c}(G)\leq \vert V(G)\vert$.
\end{conjecture}

Note that the above results imply that this conjecture holds for
bipartite graphs with maximum degree at most $6$.
Moreover, 
in \cite{AsratianCasselgrenPetrosyan} 
we 
proved that every bipartite graph $G$ where all vertex
degrees are in the set $\{1,2,4,6,7,8\}$ has a cyclic interval coloring. This result implies that for every bipartite graph $G$ with $\Delta(G)\leq 8$, we have $\mathrm{def}_{c}(G) \leq |V_{3}| + |V_{5}|$.
Thus Conjecture \ref{conj:upperbound} holds for any bipartite graph
with maximum degree at most $8$.

Next, we consider biregular bipartite graphs; such
graphs have been conjectured to always admit cyclic interval colorings
\cite{CarlJToft}.
Corollary \ref{mycorollary7} implies that $(2r-3, 2r)$-biregular and $(2r-3,2r-1)$-biregular graphs satisfy Conjecture \ref{conj:upperbound}. We shall
use the following proposition for establishing that Conjecture 
\ref{conj:upperbound} holds for $(2r-4, 2r)$-biregular graphs; 
since any bipartite graph with maximum degree $8$ satisfies 
Conjecture \ref{conj:upperbound}, we assume that 
$r\geq 5$.



\begin{proposition}
\label{prop:bipdiff4}
	If $G$ is a bipartite graph where, for some $r \geq 5$, all vertex degrees are
	in the set $\{2r-4, 2r-3,2r-2, 2r-1, 2r\}$ and 
	$$\frac{|V_{2r-4}|}{|V_{2r-2}| + |V_{2r-1}| + |V_{2r}|}  
	\leq \frac{r-1}{r-5},$$
	then $\mathrm{def}_{c}(G) \leq |V(G)|$.
\end{proposition}

\begin{proof}
	Let $G'$ be the multigraph obtained by 
	adding two loops at a vertex of degree $2r-4$ in $G$ and
	one loop at any vertex of degree $2r-3$ or $2r-2$ in $G$.
	Let $H$ be the multigraph obtained from two copies $G_1$ and $G_2$ of
	$G'$ where any vertex of odd degree in $G_1$ is joined by an edge
	to its corresponding vertex in $G_2$. Clearly, $H$ is $2r$-regular.
	
	Proceeding as in the proof of Proposition \ref{prop:upperbound},
	let $\alpha$ be the cyclic interval $2r$-coloring of $H$
	obtained by decomposing $H$ into $2$-factors and properly edge coloring
	each factor of $H$ by $2$ consecutive colors. 
	
	Denote by $\varphi$ the restriction of $\alpha$ to $G$. Then
	for any vertex $v$ of $G$,
	$S_{G}(v, \varphi)$ is a cyclic interval modulo $2r$
	if $v$ has degree $2r-2$, $2r-1$ or $2r$ in $G$. Moreover,
	if $v$ has degree $2r-3$ then $S_{G}(v, \varphi)$ has
	cyclic deficiency at most $1$, and if $v$ has degree $2r-4$,
	then $S_{G}(v, \varphi)$ has cyclic deficiency at most $2$.
	
	Consider a vertex $v$ of degree $2r-4$ in $G$.
	If for some $i \in \{1,\dots, 2r\}$,
	the colors in $\{i, i+1, i+2, i+3 \}$ (where numbers are
	taken modulo $2r$) do not appear in $S_{G}(v, \varphi)$,
	then $S_{G}(v, \varphi)$ is a cyclic interval modulo $2r$.
	Thus we may assume that at least $\frac{r}{\binom{r}{2}} |V_{2r-4}|$
	vertices $v$ in $V_{2r-4}$ satisfy that $S_{G}(v, \varphi)$
	is a cyclic interval modulo $2r$. Hence, if
	\begin{equation}
	\label{eq:numvertices}
	|V_{2r-4}| \left(1-\frac{2}{r-1}\right) \leq 
	\frac{2}{r-1}|V_{2r-4}| + |V_{2r-2}| + |V_{2r-1}| + |V_{2r}|,
	\end{equation}
	then the number of vertices with cyclic deficiency zero is at least
	the number of vertices wich cyclic deficiency two, which implies
	that $\mathrm{def}_{c}(G) \leq |V(G)|$.
	Now, \eqref{eq:numvertices} holds by assumption, so the required
	result follows.
	~$\square$
\end{proof}

\begin{corollary}
\label{cor:biregular}
		If $G$ is a $(2r-4, 2r)$-biregular graph,
		then $\mathrm{def}_{c}(G) \leq |V(G)|$.
\end{corollary}
\begin{proof}
	Let $X$ and $Y$ be the parts of $G$, where vertices in $X$ have degree
	$2r-4$.
	Since $G$ is $(2r-4,2r)$-biregular,
	we have $(2r-4)|X|  = 2r|Y|$, which implies
	that $$\frac{|X|}{|Y|} \leq \frac{r}{r-2},$$
	and so the result follows from Proposition \ref{prop:bipdiff4}.
	~$\square$
\end{proof}
\bigskip


In the following we shall present some further results 
supporting Conjecture \ref{conj:upperbound}. 
We begin 
by showing that any graph with maximum degree at most $5$
satisfies this conjecture. 
We shall use the following observation, the proof is just
straightforward case analysis and is left to the reader.

\begin{lemma}
	\label{lem:6coloring}
		If $\varphi$ is a proper $6$-edge coloring of a graph $G$ 
		with maximum degree $5$,
		then for every vertex $v$ of $G$,
		$S_G(v, \varphi)$ is near-cyclic modulo $6$ unless
		$$S_G(v,\varphi) \in \left\{\{1,4\}, \{2,5\}, 
	\{3,6\}, \{1,3,5\}, \{2,4,6\} \right\}.$$
	
\end{lemma}

\begin{theorem}
\label{prop:maxdegree5}
	If $G$ is a graph with maximum degree at most $5$, 
	then $\mathrm{def}_{c}(G)\leq |V(G)|$.
\end{theorem}

\begin{proof}
The proposition is true in the case when $\Delta(G)\leq 3$, since
any such graph $G$ admits a cyclic interval coloring 
\cite{Nadol}.  

If $\Delta(G)=4$,
consider  a proper $5$-edge coloring $\alpha$ of $G$; such a coloring
exists by Theorem  \ref{th:Vizing}. For a vertex $v$
of degree $1$ or $4$, clearly $S_G(v, \alpha)$ is a cyclic
interval modulo $5$. Moreover, straightforward case analysis
shows that for any vertex $v$ of degree $2$ or $3$ in $G$,
$S_G(v,\alpha)$ is near-cyclic. Hence,
$\mathrm{def}_{c}(G) \leq |V(G)|$ for any graph $G$ of maximum degree $4$.

Let us now consider a graph $G$ with maximum degree $5$.
		Let $M$ be a maximum matching in $G_5$, where $G_5$ is the subgraph
		of $G$ induced by the vertices of degree $5$ in $G$.
		Then, by Theorem \ref{th:Fournier}, 
		$G-M$ is $5$-edge colorable; let $\varphi$ 
		be a proper $5$-edge coloring of $G$.
		We define some subsets of $V(G)$:
		
		\begin{itemize}
			
			\item $A_{i,j,k}(\varphi)$ is the set of all vertices $v$
			for which $S_G(v, \varphi) = \{i,j,k\}$,
			where $1\leq i <j < k \leq 5$;
			
			\item $A_{i,j}(\varphi)$ is the set of all vertices $v$
			for which $S_G(v, \varphi) = \{i,j\}$, where
			$1\leq i < j\leq 5$.
		
		\end{itemize}
		
		If $|A_{1,3,5}(\varphi)| > |A_{2,3,4}(\varphi)|$,
		then by permuting the colors of $\varphi$ we can construct
		a coloring $\varphi'$ such that
		$|A_{1,3,5}(\varphi')| \leq |A_{2,3,4}(\varphi')|$.
		Thus we may assume that
		$|A_{1,3,5}(\varphi)| \leq |A_{2,3,4}(\varphi)|$.
		
		Now consider the vertices in $A_{1,4}(\varphi)$ and
		$A_{2,5}(\varphi)$. If
		$$|A_{1,4}(\varphi)|  +  |A_{2,5}(\varphi)| >
		|A_{1,2}(\varphi)| + |A_{4,5}(\varphi)|,$$
		then by permuting colors $1$ and $5$
		we obtain a coloring $\varphi'$ satisfying
		$$|A_{1,4}(\varphi')|  +  |A_{2,5}(\varphi')| <
		|A_{1,2}(\varphi')| + |A_{4,5}(\varphi')|,$$
		and, moreover,
		$|A_{1,3,5}(\varphi')| = |A_{1,3,5}(\varphi)|$
		and $|A_{2,3,4}(\varphi')| = |A_{2,3,4}(\varphi)|$.
		
		From the preceding paragraphs we thus conclude that there is 
		a proper $5$-edge coloring $\alpha$ of $G-M$
		such that $|A_{1,3,5}(\alpha)| \leq |A_{2,3,4}(\alpha)|$
		and $$|A_{1,4}(\alpha)|  +  |A_{2,5}(\alpha)| \leq
		|A_{1,2}(\alpha)| + |A_{4,5}(\alpha)|.$$
	
		We now define a proper edge coloring $\alpha'$ by
		setting $\alpha'(e) = 6$ if $e\in M$ and
		$\alpha'(e) = \alpha(e)$ if $e \notin M$.
		By Lemma \ref{lem:6coloring},
		for every vertex $v$ of $G$,
		$S_G(v, \alpha')$ is near-cyclic modulo $6$ unless
		$$S_G(v,\alpha') \in \{\{1,4\}, \{2,5\}, 
		\{3,6\}, \{1,3,5\}, \{2,4,6\} \}.$$
		Clearly, no vertex $v$ in $G$ satisfies 
		$S_{G}(v,\alpha') \in \{\{3,6\}, \{2,4,6\}\}$.
		Thus, a vertex $v$ in $G$ has cyclic deficiency $2$ under $\alpha'$
		if and only if $$S_G(v,\alpha') \in \{\{1,4\}, \{2,5\}, \{1,3,5\}\}.$$
		Now, since
		$|A_{1,3,5}(\alpha')| \leq |A_{2,3,4}(\alpha')|$,
		$$|A_{1,4}(\alpha')|  +  |A_{2,5}(\alpha')| \leq
		|A_{1,2}(\alpha')| + |A_{4,5}(\alpha')|,$$
		and the cyclic deficiencies of the sets $\{2,3,4\}, \{1,2\}$
		and $\{4,5\}$ are all equal to $0$,
		it follows that $\mathrm{def}_{c}(G)\leq |V(G)|$. ~$\square$
		
\end{proof}

\bigskip
\bigskip

All regular graphs trivially have cyclic interval colorings,
as do also all Class 1 graphs $G$ with $\Delta(G) - \delta(G) \leq 1$.
Moreover, any Class 2 graph $G$
with $\Delta(G) - \delta(G) \leq 1$ satisfies Conjecture
\ref{conj:upperbound}. The next proposition shows that a slightly
stronger statement is true.

\begin{theorem}
\label{prop:smalldiff}
	For any graph $G$ with $\Delta(G) - \delta(G) \leq 2$,
	$\mathrm{def}_{c}(G)\leq |V(G)|$.
\end{theorem}

\begin{proof} 
	By the above remark, we may assume that $\Delta(G) = \delta(G) +2$.
	Set $k =\Delta(G)$.
	
	If $G$ is Class 1, then it clearly has a proper $k$-edge
	coloring $\alpha$ such that for any vertex $S_G(v,\alpha)$ is 
	near-cyclic, implying that $\mathrm{def}_{c}(G)\leq |V(G)|$.
	Assume, consequently, that $G$ is Class $2$. 
	
	
	Let 
	$M$ be a maximum matching of $G[V_{k}]$.
	Set $H = G-M$. Note that in $H$ no pair of vertices
	of degree $k$ are adjacent. Let $M'$ be a minimum matching in $H$
	covering all vertices of degree $k$ in $H$; such a matching
	exists since $H$ is Class 1. Note that the graph  $J=H-M'$ has maximum
	degree at most $k-1$. Let $M''$ be a maximum matching in $J_{k-1}$,
	where $J_{k-1}$ is the subgraph of $J$ induced by the vertices of 
	degree $k-1$ in $J$.  Let $\hat M = M \cup M' \cup M''$.
	

	\begin{claim}
	\label{cl:GM}
				The graph $G[\hat M]$ is $2$-edge-colorable.
	\end{claim}
	\begin{proof}
		We first prove that $G[\hat M]$ has maximum degree $2$.
		Since $M$ is a maximum matching in $G[V_k]$, no edge of $M$
		is adjacent to an edge of $M'$. This implies that any vertex
		of $G$ is incident with at most two edges from $\hat M$.
		
		Now we prove that $G[\hat M]$ has no odd cycle. Since $H = G-M$ contains
		no pair of adjacent vertices of degree $k$, and every edge
		of $M'$ is incident with a vertex of degree $k$ in $H$,
		one of the ends of an edge in $M'$ has degree 
		at most $k-2$ in $J = H- M'$.
		Since $M''$ is a maximum matching in $J_{k-1}$, this means
		that no edge of $M'$ can be in a cycle of $G[\hat M]$.
		Thus, if $G[\hat M]$ contains a cycle $C$, then 
		$E(C) \subseteq M \cup M''$. However, $M$ and $M''$ are both matchings,
		so edges in $C$ lie alternately in $M$ and $M''$. Therefore
		$G[\hat M]$ contains no odd cycle.
	~$\square$
	\end{proof}
	
	\begin{claim}
	\label{cl:G-M}
				The graph $G - \hat M$ is $(k-1)$-edge-colorable.
	\end{claim}
	\begin{proof}
		As pointed out above, the graph $J = G- M \cup M'$ has maximum degree
		$k-1$, so the graph $G-\hat M$ has maximum degree at most $k-1$.
		If $G-\hat M$ has maximum degree $k-2$, then it is clearly
		$(k-1)$-edge-colorable. Suppose instead that $G - \hat M$ has maximum
		degree $k-1$. Since $M''$ is a maximum matching in $J_{k-1}$,
		no pair of vertices of degree $k-1$ are adjacent in $G-\hat M$.
		Hence, by Theorem \ref{th:Fournier}, $G$ is Class 1.
		~$\square$
	\end{proof}
	
	We continue the proof of Proposition \ref{prop:smalldiff}.
	Let $\psi$ be a proper $(k-1)$-edge coloring of $G -\hat M$
	using colors $1,\dots k-1$, and let $\varphi$ be a proper 
	$2$-edge coloring of $G[\hat M]$ using colors $k$ and $k+1$.
	Denote by $\alpha$ the coloring of $G$ that $\psi$ and $\varphi$
	yield.
	Denote by $U$ the set of vertices of degree $k-2$ in $G$
	that are adjacent to an edge of $M'$. If $v \in U$,
	then $\mathrm{def}_{c}(v,\alpha) \leq 2$, and if $v \in V_{k-2} \setminus U$,
	then $\mathrm{def}_{c}(v,\alpha) \leq 1$. Moreover, if $v \in V_{k-1}$,
	then $\mathrm{def}_{c}(v,\alpha) \leq 1$, and if $v \in V_{k}$,
	then $\mathrm{def}_{c}(v,\alpha) = 0$.
	Hence $$\mathrm{def}_{c}(G) \leq 2|U| + |V_{k-2} \setminus U|  + |V_{k-1}|.$$
	Since $|U| \leq |V_{k}|$, we deduce that
	$$\mathrm{def}_{c}(G) \leq |V_{k-2}| + |V_{k-1}|  + |V_{k}| = |V(G)|.$$
	 ~$\square$
	\end{proof}
\bigskip




\begin{acknowledgement}
We would like to thank Hrant Khachatrian for drawing the main figures in the paper. 
\end{acknowledgement}


\begin{thebibliography}{99}

\bibitem{AltinCaporHertz} S. Altinakar, G. Caporossi, A. Hertz, A comparison of integer and constraint programming models for the deficiency problem, Computers and Oper. Res. 68 (2016) 89-96.

\bibitem{AsratianCasselgrenPetrosyan} A.S. Asratian, C.J. Casselgren, P.A. Petrosyan,
Some results on cyclic interval edge colorings of graphs, Journal of Graph Theory, 2017, (in press).

\bibitem{AsrKam} A.S. Asratian, R.R. Kamalian, Interval colorings of edges of a
multigraph, Appl. Math. 5 (1987) 25-34 (in Russian).

\bibitem{AsrKamJCTB} A.S. Asratian, R.R. Kamalian, Investigation on interval
edge-colorings of graphs, J. Combin. Theory Ser. B 62 (1994) 34-43.

\bibitem{AsrDenHag} A.S. Asratian, T.M.J. Denley, R. Haggkvist, Bipartite Graphs and their Applications, Cambridge University Press, Cambridge, 1998.

\bibitem{BeinekeWilson} L.W. Beineke, R.J. Wilson, On the edge-chromatic number of a graph, Discrete Math. 5 (1973) 15-20.

\bibitem{B-OD-BHal} M. Borowiecka-Olszewska, E. Drgas-Burchardt, M. Ha\l uszczak, On the structure and deficiency of $k$-trees with bounded degree, Discrete Appl. Math. 201 (2016) 24-37.

\bibitem{B-OD-B} M. Borowiecka-Olszewska, E. Drgas-Burchardt, The deficiency of all generalized Hertz graphs and minimal consecutively non-colourable graphs in this class, Discrete Math. 339 (2016) 1892-1908.

\bibitem{BouchHertzDesau} M. Bouchard, A. Hertz, G. Desaulniers, Lower bounds and a tabu search algorithm for the minimum deficiency problem, J. Comb. Optim. 17 (2009) 168-191.

\bibitem{CarlJToft} C.J. Casselgren, B. Toft, On interval edge colorings of biregular bipartite graphs
with small vertex degrees, J. Graph Theory 80 (2015), 83-97.

\bibitem{CasselgrenPetrosyanToft} C.J. Casselgren, P.A. Petrosyan, B. Toft, 
On interval and cyclic interval edge colorings of $(3,5)$-biregular graphs,
Discrete Math. 340 (2017), 2678-2687.

\bibitem{DanDevPra} A. Daneshkhah, A. Devillers, Ch. E. Praeger, Symmetry properties of subdivision graphs, Discrete Math. 312 (2012), 86-93.

\bibitem{FengHuang} Y. Feng, Q. Huang, Consecutive edge-coloring of the generalized
$\theta$-graph, Discrete Appl. Math. 155 (2007) 2321-2327.

\bibitem{Fournier} J. C. Fournier, Coloration des aretes dun graphe, \emph{Cahiers du CERO} (Bruxelles) 15 (1973) 311--314.

\bibitem{GiaroKubale1} K. Giaro, M. Kubale, Consecutive edge-colorings of complete and incomplete Cartesian products of graphs, Cong. Num. 128 (1997) 143-149.

\bibitem{GiaroKubale2} K. Giaro, M. Kubale, Compact scheduling of zero-one time operations
in multi-stage systems, Discrete Appl. Math. 145 (2004) 95-103.

\bibitem{GiaroKubaleMalaf1} K. Giaro, M. Kubale, M. Ma\l afiejski, On the deficiency of bipartite graphs, Discrete Appl. Math. 94 (1999) 193-203.

\bibitem{GiaroKubaleMalaf2} K. Giaro, M. Kubale, M. Ma\l afiejski, Consecutive colorings of the edges of general graphs, Discrete Math. 236 (2001) 131-143.

\bibitem{Hansen} H.M. Hansen, Scheduling with minimum waiting periods, MSc Thesis, Odense University, Odense, Denmark, 1992 (in Danish).

\bibitem{HansonLotenToft} D. Hanson, C.O.M. Loten, B. Toft, On interval colorings of bi-regular bipartite graphs, Ars Combin. 50 (1998) 23-32.

\bibitem{JensenToft} T.R. Jensen, B. Toft, Graph Coloring problems, Wiley Interscience, 1995.

\bibitem{Kampreprint} R.R. Kamalian, Interval colorings of complete bipartite graphsand trees, preprint, Comp. Cen. of Acad. Sci. of Armenian SSR, Yerevan, 1989 (in Russian).

\bibitem{KamDiss} R.R. Kamalian, Interval edge colorings of graphs, Doctoral Thesis, Novosibirsk, 1990.

\bibitem{KamMir} R.R. Kamalian, A.N. Mirumian, Interval edge colorings of bipartite graphs of some class, Dokl. NAN RA 97 (1997) 3-5 (in Russian).

\bibitem{Kam_Trees} R.R. Kamalian, On cyclically-interval edge colorings of trees, Buletinul of Academy of Sciences of the Republic of Moldova, Matematica 1(68) (2012), 50-58.

\bibitem{KhachOuterplanar} H.H. Khachatrian, Deficiency of outerplanar graphs, Proceedings of the Yerevan State University, Physical and Mathematical Sciences, 2017, 51 (1), pp. 3-9.

\bibitem{Kubale} M. Kubale, Graph Colorings, American Mathematical Society, 2004.

\bibitem{KubaleNadol} M. Kubale, A. Nadolski, Chromatic scheduling in a
cyclic open shop, European J. Oper. Res. 164 (2005), 585-591.

\bibitem{Nadol} A. Nadolski, Compact cyclic edge-colorings of graphs, Discrete
Math. 308 (2008), 2407-2417.

\bibitem{PetrosKhacha} P.A. Petrosyan, H.H. Khachatrian, Interval non-edge-colorable bipartite graphs and multigraphs, J. Graph Theory 76 (2014), 200-216.

\bibitem{PetrosHrant} P.A. Petrosyan, H.H. Khachatrian, Further results on the deficiency of graphs, Discrete Appl. Math. 226 (2017), 117-126.

\bibitem{PetMkhitaryan} P.A. Petrosyan, S.T. Mkhitaryan, Interval cyclic edge-colorings of graphs, Discrete Math. 339 (2016) 1848-1860.

\bibitem{Schwartz} A. Schwartz, The deficiency of a regular graph, Discrete Math. 306 (2006) 1947-1954.

\bibitem{Seva} S.V. Sevast'janov, Interval colorability of the edges of a
bipartite graph, Metody Diskret. Analiza 50 (1990), 61-72 (in
Russian).

\bibitem{Vizing} V.G. Vizing, The chromatic class of a multigraph, Kibernetika 3 (1965) 29-39 (in Russian).

\bibitem{deWerraSolot} D. de Werra, Ph. Solot, Compact cylindrical chromatic scheduling,
SIAM J. Disc. Math, Vol. 4, N4 (1991), 528-534.

\bibitem{West} D.B. West, Introduction to Graph Theory, Prentice-Hall, New Jersey, 2001.

\end{thebibliography}
\end{document}